\documentclass[a4paper,12pt,leqno]{amsart}

\usepackage{amsmath}
\usepackage{amssymb}
\usepackage{amsthm}
\usepackage[english]{babel}
\usepackage{breqn}
\usepackage{cite}    
\usepackage{mathrsfs}                    
\usepackage[colorlinks=true]{hyperref} 

\theoremstyle{plain} 
\newtheorem{thm}{Theorem}[section] 
\newtheorem{cor}[thm]{Corollary} 
\newtheorem{lem}[thm]{Lemma} 
\newtheorem{prop}[thm]{Proposition}

\theoremstyle{definition} 
\newtheorem{defi}{Definition}[section] 
\newtheorem{rmk}{Remark}[section]

\newcommand{\eps}{\varepsilon}

\numberwithin{equation}{section}

\title[An Upper Bound for the Moments of $\gcd(n,u_n)$]
      {An Upper Bound for the Moments of a G.C.D.\\
      related to Lucas Sequences}\thanks{The author is funded by a Departmental Award and by an EPSRC Doctoral Training Partnership Award.}

\author[D. Mastrostefano]{Daniele Mastrostefano}
\address{Mathematics Institute, Zeeman Building, University of Warwick, Coventry, CV4 7AL, UK}
\email{Daniele.Mastrostefano@warwick.ac.uk}

     
\begin{document}

\begin{abstract}
Let $(u_n)_{n \geq 0}$ be a non-degenerate Lucas sequence, given by the relation $u_n=a_1 u_{n-1}+a_2 u_{n-2}$. Let $\ell_u(m)=lcm(m, z_u(m))$, for $(m,a_2)=1$, where $z_u(m)$ is the rank of appearance of $m$ in $u_n$.
We prove that 
$$\sum_{\substack{m>x\\ (m,a_2)=1}}\frac{1}{\ell_u(m)}\leq \exp(-(1/\sqrt{6}-\eps+o(1))\sqrt{(\log x)(\log \log x)}),$$
when $x$ is sufficiently large in terms of $\eps$, and where the $o(1)$ depends on $u$. Moreover, if $g_u(n)=\gcd(n,u_n)$, we will show that for every $k\geq 1$,
$$\sum_{n\leq x}g_u(n)^{k}\leq x^{k+1}\exp(-(1+o(1))\sqrt{(\log x)(\log \log x)}),$$
when $x$ is sufficiently large and where the $o(1)$ depends on $u$ and $k$. This gives a partial answer to a question posed by C. Sanna. As a by-product, we derive bounds on $\#\{n\leq x: (n, u_n)>y\}$, at least in certain ranges of $y$, which strengthens what already obtained by Sanna. 
Finally, we start the study of the multiplicative analogous of $\ell_u(m)$, finding interesting results.
\end{abstract}

\maketitle

\section{Introduction}
\label{section 1}
Let $(u_n)_{n\geq 0}$ be an integral linear recurrence, that is, $(u_n)_{n\geq 0}$ is a sequence of integers and there exist $a_1, \dots, a_k\in\mathbb{Z}$, with $a_k\neq 0$, such that
$$u_{n}=a_{1}u_{n-1}+\cdots+a_{k}u_{n-k},$$
for all integers $n\geq k$, with $k$ a fixed positive integer. We recall that $(u_n)_{n\geq 0}$ is said to be non-degenerate if none of the ratios $\alpha_{i}/\alpha_{j}$ $(i \neq j)$ is a root of unity, where $\alpha_{1}, \dots,\alpha_{r}\in\mathbb{C}$ are all the pairwise distinct roots of the characteristic polynomial
$$f_{u}(X)=X^{k}-a_{1}X^{k-1}-\cdots-a_{k}.$$
Moreover, $(u_n)_{n\geq 0}$ is said to be a Lucas sequence if $u_0=0, u_1=1,$ and $k=2$. We note that the Lucas sequence with $a_1=a_2=1$ is known as the Fibonacci sequence. We refer the reader to \cite[Chapter 1]{EPSW} for the basic terminology and theory of linear recurrences.

The function $g_{u}(n):=\gcd(n,u_n)$ has attracted the interest of several authors. For example, the set of fixed points of $g_{u}(n)$, or equivalently the set of positive numbers $n$ such that $n|u_n$, has been studied by Alba~Gonz\'alez, Luca, Pomerance, and Shparlinski \cite{ALPS}, under the mild hypotheses that $(u_{n})_{n\geq 0}$ is non-degenerate and that its characteristic polynomial has only simple roots. Moreover, this problem has been studied also by Andr\'e-Jeannin \cite{J}, Luca and Tron \cite{LT}, Sanna \cite{S2}, Smyth \cite{SM} and Somer \cite{SO}, when $(u_{n})_{n\geq 0}$ is a Lucas or the Fibonacci sequence.

On the other hand, Sanna and Tron \cite{S3, ST} have analysed the fiber $g_{u}(y)^{-1}$, when $(u_{n})_{n\geq 0}$ is non-degenerate and $y=1$, and when $(u_{n})_{n\geq 0}$ is the Fibonacci sequence and $y$ is an arbitrary positive integer. Moreover, the image $g_{u}(\mathbb{N})$ has been investigated by Leonetti and Sanna \cite{LS}, again when $(u_{n})_{n\geq 0}$ is the Fibonacci sequence.

Other important questions about the function $g_{u}(n)$ are related to its behaviour on average and its distribution as arithmetic function. From now on, we focus on the specific case in which $(u_n)_{n \geq 0}$ is a non-degenerate Lucas sequence with non-zero discriminant $\Delta_u = a_1^2 + 4a_2$. Otherwise, the sequence reduces to $u_n=n\alpha^{n}$, for a suitable $\alpha\in\mathbb{Z}$, and $g_u(n)=n$, for every positive integer $n$. Even in this particular situation, it is very difficult to find information on the distribution of $g_{u}(n)$, because of its oscillatory behaviour. For this reason, it is natural to consider the flatter function $\log(g_{u}(n))$, for which an asymptotic formula for its mean value, and more in general for its moments, has been given by Sanna, who proved the following theorem \cite[Theorem 1.1]{S}.
\begin{thm}
\label{thm 1.1}
Fix a positive integer $\lambda$ and some $\eps>0$. Then, for all sufficiently large $x$, how large depending on $a_1,a_2,\lambda$ and $\eps$, we have
\begin{equation}
\label{eq: 1.1}
\sum_{n\leq x}(\log g_{u}(n))^{\lambda}=M_{u,\lambda}x+E_{u,\lambda}(x),
\end{equation}
where $M_{u,\lambda}>0$ is a constant depending on $a_1,a_2$ and $\lambda$, and the error term is bounded by 
$$E_{u, \lambda}(x)\ll_{u,\lambda}x^{(1+3\lambda)/(2+3\lambda)+\eps}.$$
\end{thm}
Also, Sanna showed that the constant $M_{u,\lambda}$ can be expressed through a convergent series.

An immediate consequence of the previous result is the possibility of finding information about the distribution of $g_{u}$ \cite[Corollary 1.3]{S}.
\begin{cor}
\label{cor 1.2} For each positive integer $\lambda$, we have 
\begin{equation}
\label{eq: 1.2}
\#\{n\leq x: g_{u}(n)>y\}\ll_{u,\lambda}\frac{x}{(\log y)^{\lambda•}},
\end{equation}
for all $x,y>1$.
\end{cor}
In the same article, Sanna raised the question of finding an asymptotic formula for the moments of the function $g_{u}(n)$ itself. We are not able to answer to this apparently difficult question, but we can at least give a non-trivial estimate for them. The result is the following.
\begin{thm}
\label{thm 1.3}
For every integer $k\geq 1$ and $u_n$ a non-degenerate Lucas sequence, we have
\begin{equation}
\label{eq: 1.3}
\sum_{n\leq x} g_u(n)^{k}\leq x^{k+1}\exp\left(-\left(1+o(1)\right)\sqrt{(\log x)(\log \log x)}\right),
\end{equation}
as $x$ tends to infinity and where the $o(1)$ depends on $u$ and $k$.
\end{thm}
For each positive integer $m$ relatively prime with $a_2$, let $z_u(m)$ be the rank of appearance of $m$ in the Lucas sequence $(u_n)_{n\geq 0}$, that is, $z_u(m)$ is the smallest positive integer $n$ such that $m$ divides $u_n$. It is well known that $z_u(m)$ exists (see, e.g., \cite{R}). Also, put $\ell_u(m) :=lcm(m, z_u(m))$. There is a simple trick to relate the moments of $g_u(n)$ with the rate of convergence of the series $\sum_{m>x, (m,a_2)=1}1/ \ell_u(m)$, which has been partially studied by several authors. We will deduce a slightly weaker version of Theorem \ref{thm 1.3}, in which the constant in the exponential is replaced by $-1/\sqrt{6}+\eps+o(1)$, for every $\eps>0$, from it and the following bound.
\begin{prop}
\label{prop 1.4} 
For every non-degenerate Lucas sequence $u_n$, we have
\begin{equation}
\label{eq: 1.4}
\sum_{\substack{m>x \\(m,a_2)=1}}\frac{1}{\ell_u(m)}\leq\exp(-(1/ \sqrt{6}-\eps+o(1))\sqrt{(\log x)(\log \log x)}),
\end{equation}
when $x$ is large in terms of $\eps$ and where the $o(1)$ depends on $u$.
\end{prop}
In the proof of Proposition \ref{prop 1.4} we highlight a method, based essentially on the distribution of smooth numbers, to achieve the above bound. It seems reasonable to think that a deeper analysis of the structure of $\ell_u(n)$ could lead to understand better the behaviour of $\sum_{\substack{m>x,(m,a_2)=1}}1/\ell_u(m)$ and consequently to improve the result about the moments of $g_u(n)$. Nevertheless, using a completely different and more direct approach that we will describe later, we can obtain the stronger stated estimate in Theorem \ref{thm 1.3}.

It is immediate to deduce from Theorem \ref{thm 1.3} the following improvement on the distribution of $g_{u}(n)$ at least when $y$ varies uniformly in a certain range.
\begin{cor}
\label{cor 1.5}
We have  
\begin{equation}
\label{eq: 1.5}
\#\{ n\leq x : g_u(n)>y\}\leq \frac{x^{2}}{y\exp((1+o_{u}(1))\sqrt{(\log x)(\log \log x)})},
\end{equation}
for every $y\geq 1$, when $x$ is sufficiently large.
\end{cor}
\begin{proof}
By using \eqref{eq: 1.5} with $k=1$ we obtain
\begin{equation}
\label{eq: 1.6}
\#\{ n\leq x : g_u(n) >y\}\leq \sum_{n\leq x}\frac{g_u(n)}{y}
\end{equation}
$$\leq \frac{x^{2}}{y\exp((1+o_{u}(1))\sqrt{(\log x)(\log \log x)})},$$
for every $y\geq 1$.
\end{proof}
We observe that this is an improvement of \eqref{eq: 1.2}, only for certain values of $y$, e.g. like for those satisfying 
\begin{equation}
\label{eq: 1.7}
x\exp(-(1/2+o_{u}(1))\sqrt{(\log x)(\log \log x)})\leq y\leq x.
\end{equation}
Consider now the multiplicative function $L_u(n)$ such that $L_u(p^{k})=\ell_u(p^{k})$, for every prime number $p\nmid a_2$ and power $k\geq 1$, and $L_u(p^{k})=p^{k}$, otherwise.
Using arguments coming from the theory of Dirichlet series of multiplicative functions, we end up with the following estimate.
\begin{prop}
\label{prop 1.6}
For every $u_n$ non-degenerate Lucas sequence, we have
\begin{equation}
\label{eq: 1.8}
\sum_{\substack{n>x}}\frac{1}{L_u(n)}\ll_u x^{-1/3+\eps},
\end{equation}
for every $\eps>0$, when $x$ is sufficiently large with respect to $\eps$.
\end{prop}
The above result shows that the lack of multiplicativity of $\ell_u(n)$ is the principle cause for the weaker upper bound in \eqref{eq: 1.4}.

\section{Notations}
For a couple of real functions $f(x), g(x)$, with $g(x)>0$, we indicate with $f(x)=O(g(x))$ or $f(x)\ll g(x)$ that there exists an absolute constant $c>0$ such that $|f(x)|\leq cg(x)$, for $x$ sufficiently large. When the implicit constant $c$ depends from a parameter $\alpha$ we indicate the above bound with $f(x)\ll_{\alpha} g(x)$ or equivalently with $f(x)=O_{\alpha}(g(x))$.

Throughout, the letter $p$ is reserved for a prime number. We write $(a,b)$ and $[a,b]$ to denote the greatest common divisor and the least common multiple of integers $a,b$. As usual, we denote with $\lfloor w\rfloor$ the integer part of a real number $w$ and we indicate with $P(n)$ the greatest prime factor of a positive integer $n$.
\section{Preliminaries}
We begin by recalling the definition of the Jordan's totient function.
\begin{defi}
\label{def 3.1} The Jordan's totient function of degree $k$ is defined as
$$J_{k}(n)=n^{k}\prod_{p\mid n}\left(1-\frac{1}{p^{k}•}\right),$$
for every $k\geq 1$ and natural integers $n$.
\end{defi}
Clearly, $J_{1}(n)=\varphi(n)$, the Euler's totient function, and it is immediate to see that $J_k(n)$ verifies the following identity.
\begin{lem} We have 
\begin{equation}
\label{eq: 3.1}
n^{k}=\sum_{d\mid n}J_{k}(d),
\end{equation}
for any $k\geq 1$ and natural integers $n$.
\end{lem}
The next lemma summarizes some basic properties of $\ell_u(n)$ and $z_u(n)$, which we will implicitly use later without further mention.
\begin{lem}
\label{lem 3.2}
For all positive integers $m$, $n$ and all odd prime numbers $p$, we have:
\begin{enumerate}
\item $m \mid u_n$ if and only if $z_u(m) \mid n$ and $(m,a_2)=1$.
\item $z_u([m, n]) = [z_u(m), z_u(n)]$, whenever $(mn,a_2)=1$.
\item $m \mid \gcd(n, u_n)$ if and only if $(m, a_2)=1$ and $\ell_u(m) \mid n$.
\item $\ell_u([m, n]) = [\ell_u(m), \ell_u(n)]$, whenever $(mn,a_2)=1$.
\item $\ell_u(p^{j}) = p^{j} z_u(p)$ if $p\nmid \Delta_u,$ and $\ell_u(p^{j}) = p^{j}$ if $p\mid\Delta_u$, for every $p\nmid a_2$ and $j\geq 1$.
\item $z_u(p)|p\pm 1$, if $p\nmid \Delta_u,$ and $z_u(p)=p$ if $p\mid\Delta_u$, for every $p\nmid a_2$.
\end{enumerate}
\end{lem}
For any $\gamma>0$, let us define
$$\mathcal{Q}_{u,\gamma}:=\{p: p\nmid a_2, z_u(p)\leq p^{\gamma}\}.$$
The following is \cite[Lemma 2.1]{ALPS}.
\begin{lem}
\label{lem 3.3}
For all $x^{\gamma},y\geq 2$ and for any non-degenerate Lucas sequence $(u_n)_{n\geq 0}$, we have 
$$\#\{p: z_u(p)\leq y\}\ll_u \frac{y^{2}}{\log y},\ \ \mathcal{Q}_{u,\gamma}(x)\ll_u \frac{x^{2\gamma}}{\gamma\log x}.$$
\end{lem}
It has been proven by Sanna and Tron \cite[Lemma 3.2]{ST} that the series $\sum_{(n, a_2)=1}1/\ell_u(n)$ converges. We consider the following identity
\begin{equation}
\label{eq: 3.2}
\sum_{\substack{n>x \\ (n,a_2)=1}}\frac{1}{\ell_u(n)}=\sum_{\substack{n>x\\ P(n)>y\\ (n,a_2)=1}}\frac{1}{\ell_u(n)}+\sum_{\substack{n>x\\ P(n)\leq y\\ (n,a_2)=1}}\frac{1}{\ell_u(n)}.
\end{equation}
We note that the first sum in the right hand side of \eqref{eq: 3.2} has been already investigated by Sanna \cite[Lemma 2.5]{S} and we report here the result which he obtained.
\begin{prop} 
\label{prop 3.4}
We have 
\begin{equation}
\label{eq: 3.3}
\sum_{\substack{(m,a_2)=1\\P(m)>y}}\frac{1}{\ell_u(m)}\ll_u\frac{1}{y^{1/3-\eps}•},
\end{equation}
for all $\eps\in(0, 1/4]$ and $y\gg_{u,\eps}1.$
\end{prop}
Regarding the second sum in the right hand side of \eqref{eq: 3.2} we provide an estimate in the next lemma.
\begin{lem}
\label{lem 3.5}
Supposing that $y>(\log x)^{2}$ and $v=\log x/\log y$ tends to infinity as $x$ tends to infinity, we have 
\begin{equation}
\label{eq: 3.4}
\sum_{\substack{n>x\\ P(n)\leq y\\ (n, a_2)=1}}\frac{1}{\ell_u(n)}\ll_u (\log y)e^{-\sqrt{y}/2\log y}+\frac{\log y}{\log v}e^{-v\log v}.
\end{equation}
\end{lem}
\begin{proof} Since $\ell_u(n)\geq n$, we may write
\begin{equation*}
\sum_{\substack{n>x\\ P(n)\leq y\\ (n, a_2)=1}}\frac{1}{\ell_u(n)}\leq \int_{x}^{\infty}\frac{d\psi(t,y)}{t},
\end{equation*}
where $\psi(t,y)$ is the counting function of the $y$-smooth numbers less than $t$. Clearly, we have
\begin{equation}
\label{eq: 3.5}
\int_{x}^{\infty}\frac{d\psi(t,y)}{t}=\frac{\psi(t,y)}{t}\bigg|_{x}^{\infty}+\int_{x}^{\infty}\frac{\psi(t,y)}{t^{2}}dt.
\end{equation}
To estimate the second term on the right hand side of \eqref{eq: 3.5} we suppose first that $y>\log^{2}(x)$ and then we split it into two parts:
\begin{equation*}
\int_{x}^{\infty}\frac{\psi(t,y)}{t^{2}}dt=\int_{x}^{z}\frac{\psi(t,y)}{t^{2}}dt+\int_{z}^{\infty}\frac{\psi(t,y)}{t^{2}}dt,
\end{equation*}
where we put $z=e^{\sqrt{y}}$. Using the estimate \cite[Theorem 1, \S 5.1, Chapter III]{T}
\begin{equation}
\label{eq: 3.6}
\psi(t,y)\ll te^{-\log t/2\log y}=t^{1-1/2\log y},
\end{equation}
valid uniformly for $t\geq y\geq 2$, we obtain
\begin{equation}
\label{eq: 3.7}
\int_{z}^{\infty}\frac{\psi(t,y)}{t^{2}}\ll \int_{z}^{\infty}t^{-1-1/(2\log y)}dt\ll (\log y) z^{-1/(2\log y)}=(\log y)\exp\left(-\frac{\sqrt{y}}{2\log y}\right).
\end{equation}
By the Corollary of the Theorem 3.1 in \cite{CEP}, we know that 
$$\psi(t,y)\leq t\exp\left(-(1+o(1))\frac{\log t}{\log y}\log\left(\frac{\log t}{\log y}\right)\right),$$ 
in the region $y>\log^{2}t$. Here the $o(1)$ is with respect to $\log t/\log y\rightarrow \infty$. If $v=\log x/ \log y$ tends to infinity as $x$ tends to infinity, then we may use the simpler 
\begin{equation}
\label{eq: 3.8}
\psi(t,y)\leq t\exp\left(-\frac{\log t}{\log y}\log\left(\frac{\log t}{\log y}\right)\right),
\end{equation}
for any $x\leq t\leq z$. Note that equation \eqref{eq: 3.8} also follows from the aformentioned Corollary in \cite{CEP}. Let us suppose to be in this situation. Now, inserting this bound and using the change of variable $s=\log t$, we get
$$\int_{x}^{z}\frac{\psi(t,y)}{t^{2}}dt\leq \int_{\log x}^{\sqrt{y}}\exp\left(-\frac{s}{\log y}\log\left(\frac{s}{\log y}\right)\right)ds,$$
which after another change of variable $s=w\log y$ becomes
$$(\log y)\int_{\log x/\log y}^{\sqrt{y}/\log y}\exp(-w\log w)dw.$$
Using that $w\geq v$ and putting $w\log v=r$, we find
\begin{equation}
\label{eq: 3.9}
\int_{x}^{z}\frac{\psi(t,y)}{t^{2}}dt\leq \frac{\log y}{\log v}\int_{v\log v}^{\sqrt{y}\log v/\log y}e^{-r}dr\leq \frac{\log y}{\log v}e^{-v\log v}.
\end{equation}
Regarding the first term on the right hand side of \eqref{eq: 3.5}, we note that
$$\frac{\psi(t,y)}{t}\bigg|_{x}^{\infty}\leq \lim_{t\rightarrow\infty}\frac{\psi(t,y)}{t}\ll \lim_{t\rightarrow\infty}t^{-1/2\log y}=0,$$
by \eqref{eq: 3.6}. Collecting the results, we obtain the estimate \eqref{eq: 3.4}.
\end{proof}\
\\
Finally, we can deduce the stated estimate on $\sum_{n>x}1/\ell_u(n)$.
\begin{proof}[Proof of Proposition \ref{prop 1.4}] 
By Proposition \ref{prop 3.4} and Lemma \ref{lem 3.5} we conclude that
$$\sum_{\substack{n>x\\ (n,a_2)=1}}\frac{1}{\ell_u(n)}\ll_u \frac{1}{y^{1/3-\eps}}+\frac{\log y}{\log v}e^{-v\log v},$$
for every $\eps>0$, if $y$ is sufficiently large in terms of $\eps$. It is immediate to see that the best choice for $y$ is of the form $y=\exp(C\sqrt{(\log x)(\log\log x)})$, with $C$ a suitable positive constant to be chosen later. After some easy considerations, we obtain 
$$\sum_{\substack{n>x\\(n,a_2)=1}}\frac{1}{\ell_u(n)}\ll_u \exp\left(-C(1/3-\eps)\sqrt{(\log x)(\log \log x)}\right)$$
$$+\exp\left(-\frac{1}{2C} (1-o(1))\sqrt{(\log x)(\log \log x)}\right),$$
where $o(1)$ tends to zero from the right as $x$ goes to infinity. Now, choosing $C=1/\sqrt{2(1/3-\eps)}$, we see that
$$\sum_{\substack{n>x\\ (n,a_2)=1}}\frac{1}{\ell_u(n)}\ll_u \exp\left(-\frac{(1-o(1))(1-\eps)}{\sqrt{6}}\sqrt{(\log x)(\log \log x)}\right),$$
for every $\eps>0$ and $x$ sufficiently large with respect to $\eps$.
\end{proof}

\section{Proof of weak version of Theorem 1.3}

\begin{proof}
We start inserting equation \eqref{eq: 3.1} inside our main sums.
\begin{equation}
\label{eq: 4.1}
\sum_{n\leq x}(n, u_{n})^{k}=\sum_{n\leq x}\sum_{\substack{d\mid (n, u_n)}} J_{k}(d)=\sum_{d\leq x}J_{k}(d)\sum_{\substack{n\leq x\\ d\mid (n, u_n)}}1=\sum_{\substack{d\leq x\\ (d,a_2)=1}}J_{k}(d)\sum_{\substack{n\leq x\\ \ell_u(d)\mid n }}1,
\end{equation}
by part (3) of Lemma \ref{lem 3.2}. Clearly, the last sum in \eqref{eq: 4.1} is
\begin{equation}
\label{eq: 4.2}
\sum_{\substack{d\leq x\\ (d, a_{2})=1}}J_{k}(d)\bigg\lfloor \frac{x}{\ell_u(d)}\bigg\rfloor\leq x\sum_{\substack{d\leq x\\ (d, a_{2})=1}}\frac{J_{k}(d)}{\ell_{u}(d)}\leq x\sum_{\substack{d\leq x\\ (d, a_{2})=1}}\frac{d^{k}}{\ell_{u}(d)}.
\end{equation}
But now we observe that
\begin{equation*}
\sum_{\substack{d\leq x\\ (d, a_{2})=1}}\frac{d^{k}}{\ell_{u}(d)}=\sum_{\substack{d\leq x^\delta \\ (d, a_{2})=1}}\frac{d^{k}}{\ell_{u}(d)}+\sum_{\substack{x^\delta <d\leq x\\ (d, a_{2})=1}}\frac{d^{k}}{\ell_{u}(d)}
\end{equation*}
$$\ll x^{k\delta}+x^{k}\sum_{\substack{d>x^\delta\\ (d, a_{2})=1}}\frac{1}{\ell_{u}(d)}$$
$$\ll x^{k}\exp(-(1/\sqrt{6}-\eps+o(1))\sqrt{\delta}\sqrt{(\log x)(\log \log x)}),$$
for any $\delta\in (0,1)$, using that the series $\sum_{n}1/\ell_u(n)$ converges and the bound \eqref{eq: 1.4}, and for any $x$ large in terms of $\delta$ and $\eps$. Now, choosing $\delta$ close to 1 as a function of $\eps$, and by the arbitrarity of $\eps$, we find
\begin{equation}
\label{eq: 4.3}
\sum_{\substack{d\leq x\\ (d, a_{2})=1}}\frac{d^{k}}{\ell_{u}(d)}\leq x^{k}\exp(-(1/\sqrt{6}-\eps+o(1))\sqrt{(\log x)(\log \log x)}),
\end{equation} 
where the $o(1)$ depends on $u$ and $k$ and $x$ is chosen large enough with respect to $\eps$. Inserting \eqref{eq: 4.3} in \eqref{eq: 4.2} and \eqref{eq: 4.2} in \eqref{eq: 4.1}, the proof is finished.
\end{proof}

\section{Proof of Theorem 1.3} 
\begin{proof}
Let $y:=\exp(\frac{1}{2}\sqrt{(\log x)(\log\log x)})$. We define a partition of $\{n: n\leq x\}$, by setting
\begin{equation*}
\begin{array}{lllll}
E_{1}(x)=\{n\leq x: P(n)\nmid u_n\};\\
\\
E_2(x)=\{n\leq x: P(n)\leq y\};\\
\\
E_3(x)=\{n\leq x: P(n)>y^{6},\ P(n)\in Q_{u,1/3}(x)\};\\
\\
E_4(x)=\{n\leq x: P(n)>y^{6},\ P(n)\not\in Q_{u,1/3}(x)\};\\
\\
E_5(x)=\{n\leq x\}\setminus E_1\cup E_2\cup E_3\cup E_4.\end{array}
\end{equation*}
Let $S_i=\sum_{n\in E_i(x)}(n, u_n)^{k}$, for every $i=\{1,2,3,4,5\}$. We note that if $n\in E_1(x)$, then $(n,u_n)|(n/P(n))$ and we deduce that
\begin{equation}
\label{eq: 5.1}
S_1\leq \sum_{n\leq x}\left(\frac{n}{P(n)}\right)^{k}\leq x^{k}\sum_{n\leq x}\frac{1}{P(n)^{k}}\leq x^{k+1}\exp((-\sqrt{2k}+o(1))\sqrt{(\log x) (\log\log x)}),
\end{equation}
where the last inequality follows by \cite[equation 1.6]{IP}. Moreover, it is immediate to see that
\begin{equation*}
S_2\leq x^{k}\psi(x,y)\leq x^{k+1}\exp(-(1+o(1))u\log u), 
\end{equation*}
by the Corollary of Theorem 3.1 in \cite{CEP}, where $u=\log x/\log y$ and $o(1)$ tends to zero as $u$ tends to infinity. We observe that we can apply this result because we chose a value of $y$ sufficiently large. Notice also that by our choice of $y$ we have actually got
\begin{equation}
\label{5.2}
S_2\leq x^{k+1}\exp(-(1+o(1))\sqrt{(\log x)(\log\log x)}), 
\end{equation}
which dominates \eqref{eq: 5.1}. Regarding the third sum, we simply use $S_3\leq x^{k}\# E_3(x)$. Now, if $n\in E_3(x)$ we can factorize $n=P(n)m$, with $P(n)>y^{6}$ and $P(n)\in Q_{u,1/3}(x)$. This implies that $m<x/y^6$ and that $P(n)\in Q_{u,1/3}(x/m)$. Consequently
\begin{equation*}
\#E_3(x)\leq \sum_{m\leq x/y^6}\#Q_{u,1/3}(x/m)\ll x^{2/3}\sum_{m\leq x/y^6}\frac{1}{m^{2/3}}\ll \frac{x}{y^{2}},
\end{equation*}
by Lemma \ref{lem 3.3} and a standard final computation. This leads to
\begin{equation}
\label{5.3}
S_3\ll x^{k+1}\exp(-2\log y),
\end{equation}
which is of the same order of magnitude of \eqref{5.2}. For what concerns the fourth sum, by part $(1)$ and $(6)$ of Lemma \ref{lem 3.2}, we have that $z_u(P(n))|n$ and $z_u(P(n))|P(n)\pm 1$, implying that $P(n)z_u(P(n))|n$. Note that we can affirm the first two divisibility conditions, because we can suppose $P(n)\nmid a_2 \Delta_u$ and odd, since $y$ is large enough. We deduce that
\begin{equation*}
\#E_4(x)\leq \sum_{\substack{p>y^6 \\ p\not\in Q_{u,1/3}(x)}}\frac{x}{pz_u(p)}\leq \sum_{p>y^6}\frac{x}{p^{4/3}}\ll \frac{x}{y^{2}},
\end{equation*}
by a standard computation. Therefore, we find
\begin{equation}
\label{5.4}
S_4\leq x^{k}\# E_4(x)\ll x^{k+1}\exp(-2\log y),
\end{equation}
which coincides with \eqref{5.3}.
We are left then with the estimate of $S_{5}(x)$. To this aim we strictly follow an argument already employed in the proof of \cite[Theorem 2]{ALPS}. For any non-negative integer $j$, let $I_j:=[2^j,2^{j+1})$. We cover $I:=[y,y^{6})$ by these dyadic intervals, and we define $a_j$ via $2^j=y^{a_j}$. We shall assume the variable $j$ runs over just those integers with $I_j$ not disjoint from $I$.
For any integer $k$, define $\mathcal{P}_{j,k}$ as the set of primes $p\in I_j$ with $z_u(p)\in I_k$.  Note that, by Lemma \ref{lem 3.3}, we have $\#\mathcal{P}_{j,k}\ll 4^k$. We have
\begin{equation}
\#E_5(x)\leq\sum_j\sum_k\sum_{p\in\mathcal{P}_{j,k}}\sum_{\substack{n\leq x\\P(n)|u_n\\ P(n)=p}}1\leq \sum_j\sum_k\sum_{p\in\mathcal{P}_{j,k}}\psi\left(\frac{x}{pz_u(p)},p\right)
\end{equation}
$$\leq \sum_j\sum_k\sum_{p\in\mathcal{P}_{j,k}}\frac{x}{pz_u(p)y^{2/a_j+o(1)}},$$
as $x\to\infty$, where we have used the Corollary of Theorem 3.1 in \cite{CEP} for the last estimate. For $k>j/2$, we use the estimate
$$
\sum_{p\in\mathcal{P}_{j,k}}\frac1{pz_u(p)}\leq 2^{-k}\sum_{p\in I_j}\frac1p\leq 2^{-k}
$$
for $x$ large. For $k\le j/2$, we use the estimate
$$
\sum_{p\in\mathcal{P}_{j,k}}\frac1{pz_u(p)}\ll\frac{4^k}{2^j2^k}=2^{k-j},
$$
since there are at most order of magnitude $4^k$ such primes, as noted before.
Thus,
\begin{equation}
\sum_k\sum_{p\in\mathcal{P}_{j,k}}\frac{1}{pz_u(p)}= \sum_{k>j/2}\sum_{p\in\mathcal{P}_{j,k}}\frac{1}{pz_u(p)}+\sum_{k\le j/2}\sum_{p\in\mathcal{P}_{j,k}}\frac{1}{pz_u(p)}\ll 2^{-j/2}=y^{-a_j/2}.
\end{equation}
Collecting the above computations, we find
$$
\#E_5(x)\leq\sum_j\frac{x}{y^{a_j/2+2/a_j+o(1)}},\ \textrm{as}\ x\to\infty.
$$  
Since the minimum value of $t/2+2/t$ for $t>0$ is $2$ occuring at $t=2$, we may affirm that 
$$\#E_5(x)\leq x/y^{2+o(1)},\ \textrm{as}\ x\to\infty,$$
which leads to an estimate for $S_5$ as large as that one for $S_2$. We conclude that
$$\max\{S_1,S_2,S_3,S_4,S_5\}\leq x^{k+1}\exp(-(1+o(1))\sqrt{(\log x) (\log\log x)}),$$
proving Theorem \ref{thm 1.3}.
\end{proof}
\section{The multiplicative analogous of $\ell_u(n)$}
Let us define the multiplicative function $L_u(n)$ such that $L_u(p^{k})=\ell_u(p^{k})$, for every prime numbers $p\nmid a_2$ and power $k\geq 1$, and $L_u(p^{k})=p^{k}$, otherwise. Now, consider the Dirichlet series of the function $n/L_u(n)$, given by
$$\alpha(s)=\sum_{n\geq 1} \frac{n}{n^{s}L_u(n)•}.$$
Suppose that it converges for $s>\sigma_{c}$, where $\sigma_c$ is the abscissa of absolute and ordinary convergence of $\alpha(s)$. Certainly, since $\ell_u(n)\leq L_u(n)$, for every $n$, and since we know that the series of the reciprocals of $\ell_u(n)$ converges, we have $\sigma_{c}\leq 1$. Then, for any $s\in\mathbb{C}$ with $\Re(s)=\sigma>\sigma_{c}$ we can consider the Euler product and it converges to the Dirichlet series in this range. Therefore, we can write
$$
\alpha(s)=\prod_{p\nmid 2a_2\Delta_u}\left(1+\sum_{k\geq 1}\frac{f(p^{k})}{p^{ks}}\right)\beta(s),
$$
where $f(n)=n/L_u(n)$ and $\beta(s)$ is an analytic function in $\Re(s)>0$. Since by property (5) of Lemma \ref{lem 3.2} we find that $f(p^{k})=1/z_u(p)$, for any $k\geq 1$ and prime $p\nmid 2a_2\Delta_u$, we have 
\begin{equation}
\label{eq: 6.1}
\alpha(s)=\prod_{p\nmid 2a_2\Delta_u}\left(1+\frac{f(p)}{p^{s}}\frac{p^{s}}{p^{s}-1•}\right)\beta(s)=\prod_{p\nmid 2a_2\Delta_u}\left(1+\frac{1}{z_u(p)(p^{s}-1)}\right)\beta(s).
\end{equation}
Now, the final product in \eqref{eq: 6.1} converges if and only if 
\begin{equation*}
\sum_{p\nmid 2a_2\Delta_u}\frac{1}{z_u(p)(p^{s}-1)•}
\end{equation*}
converges. Therefore, it suffices to prove that 
$$\lim_{x\rightarrow \infty}\sum_{\substack{p>x}}\frac1{z_u(p)(p^{\sigma}-1)}=0.$$
We estimate the last sum separating between primes $p\in\mathscr{Q}_{u,\gamma}$ or $p\not\in \mathscr{Q}_{u,\gamma}$. In the first case we obtain
\begin{equation}
\label{eq: 6.2}
\sum_{\substack{p>x\\ p\in\mathscr{Q}_{u,\gamma}}}\frac{1}{z_u(p)(p^{\sigma}-1)•}\ll \int_{x}^{\infty}\frac{d(\#\mathscr{Q}_{u,\gamma}(t))}{t^{\sigma}}\ll_u \frac{1}{(\sigma-2\gamma)x^{\sigma-2\gamma}•},
\end{equation}
by Lemma \ref{lem 3.3}, if we choose $\sigma>2\gamma$. On the other hand, in the second case we get
\begin{equation}
\label{eq: 6.3}
\sum_{\substack{p>x\\p\not\in\mathscr{Q}_{u,\gamma}}}\frac{1}{z_u(p)(p^{\sigma}-1)•}\ll \sum_{p>x}\frac{1}{p^{\sigma+\gamma}•}\ll \frac{1}{(\sigma+\gamma-1)x^{\sigma+\gamma-1}},
\end{equation}
if we choose $\sigma+\gamma>1$. Comparing \eqref{eq: 6.2} with \eqref{eq: 6.3}, we are led to take $\gamma=1/3$ and we have showed that
\begin{equation}
\label{eq: 6.4}
\sum_{\substack{p>x}}\frac1{z_u(p)(p^{\sigma}-1)}\ll_u \frac{1}{\eps x^{\eps}•},
\end{equation}
if $\sigma=2/3+\eps$, for every $\eps>0$, and consequently that $\alpha(s)$ converges for every $s$ with $\Re(s)>2/3$, or equivalently that $\sigma_c\leq 2/3$. An immediate application of this result is the following. Let us define
$$F(s)=\sum_{n\geq 1}\frac{1}{n^{s}L_u(n)•}.$$
Then, $F(s)$ has the abscissa of convergence $\sigma_c '\leq -1/3$. This is equivalent to have obtained a strong bound on the tail of $F(0)$. The intermediate passage is made explicit in the next lemma (see e.g. \cite[\S 11.3, Lemma 1]{A}).
\begin{lem}
\label{lem 6.1}
Suppose that $G(s)=\sum_{n\geq 1}a_n n^{-s}$ is a Dirichlet series of a sequence $(a_n)_{n\geq 1}$ of positive real numbers, with abscissa of convergence $\sigma_c '$. Suppose that $G(0)$ converges. Then, we have $\sigma_c '=\inf\{\theta : \sum_{n>x}a_n\ll x^{\theta}\}.$
\end{lem}
Since $F(s)$ satisfies the hypotheses of the Lemma \ref{lem 6.1}, by \eqref{eq: 6.4}, we deduce that 
\begin{equation*}
\sum_{n>x}\frac{1}{L_u(n)•}\ll_u x^{-1/3+\eps},
\end{equation*}
for every $\eps>0$, proving Proposition \ref{prop 1.6}.
\begin{rmk}
We believe that a finer study of $L_u(n)$ could lead to understand better the structure of $\ell_u(n)$, though the lack of multiplicativity of the latter makes difficult its study starting with information from the former. For instance, it can be shown that the integers $n$, which have at least
two prime factors $p_1,p_2$ such that a fixed prime $q$ divides both $z_u(p_1)$ and $z_u(p_2)$, have asymptotic density $1$. Thus, when calculating $z_u(n)$ as a least common multiple, there is a cancellation of a factor $q$. In other words, for any positive real number $C$, most integers $n$ have $L_u(n)/\ell_u(n) > C$. This suggests that the two mentioned functions are not always very close each other.
\end{rmk}

\section*{Acknowledgements}
I would like to thank Carlo Sanna for suggesting this problem and for introducing me to the theory of linear recurrences. A special thanks goes also to the anonymous referee, for careful reading and useful advice.

\bibliographystyle{amsplain}

\begin{thebibliography}{9} 

\bibitem{ALPS}
J.~J. Alba~Gonz\'alez, F.~Luca, C.~Pomerance, and I.~E. Shparlinski, \emph{On
  numbers {$n$} dividing the {$n$}th term of a linear recurrence}, Proc. Edinb.
  Math. Soc. (2) \textbf{55} (2012), no.~2, 271--289.
  
\bibitem{J}
R.~Andr\'e-Jeannin, \emph{Divisibility of generalized {F}ibonacci and {L}ucas
  numbers by their subscripts}, Fibonacci Quart. \textbf{29} (1991), no.~4,
  364--366.

\bibitem{A}
T.~Apostol, \emph{Introduction to Analytic Number Theory}, Springer-Verlag, 1976.

\bibitem{CEP}
E.~R.~Canfield, P.~Erd\H{o}s and C.~Pomerance, \emph{On
a problem of Oppenheim concerning ``Factorisatio Numerorum''}, J. Number
Theory \textbf{17} No.~1, 1--28 (1983).

\bibitem{EPSW} 
G.~Everest, A.~van der Poorten, I.~Shparlinski, T.~Ward, \emph{Recurrence Sequences}, Mathematical Surveys and Monographs, vol.104, American Mathematical Society, Providence, RI, 2003.

\bibitem{IP}
A.~Ivi\'{c} and C.~Pomerance, \emph{Estimate for certain sums involving the largest prime factor of an integer}, Topics in classical number theory (Budapest, 1981), 769--789, North Holland, 1984.

\bibitem{LS}
P.~Leonetti and C.~Sanna, \emph{On the greatest common divisor of {$n$} and the
  {$n$}th {F}ibonacci number}, Rocky Mountain J. Math. (2018).
  
\bibitem{LT}
F.~Luca and E.~Tron, \emph{The distribution of self-{F}ibonacci divisors},
  Advances in the theory of numbers, Fields Inst. Commun., vol.~77, Fields
  Inst. Res. Math. Sci., Toronto, ON, 2015, pp.~149--158.

\bibitem{R} 
M.~Renault, \emph{The period, rank, and order of the (a, b)-Fibonacci sequence mod $m$}, Math. Mag. 86 (5) (2013) 372--380.

\bibitem{S} 
C.~Sanna, \emph{The moments of the logarithm of a G.C.D. related to Lucas sequences}, J. of Number Theory (2018).

\bibitem{S2}
C.~Sanna, \emph{On numbers {$n$} dividing the {$n$}th term of a {L}ucas
  sequence}, Int. J. Number Theory \textbf{13} (2017), no.~3, 725--734.
  
\bibitem{S3}
C.~Sanna, \emph{On numbers {$n$} relatively prime to the {$n$}th term of a
  linear recurrence}, Bull. Malays. Math. Sci. Soc.
  
\bibitem{ST} 
C.~Sanna, E.~Tron, \emph{The density of numbers n having a prescribed G.C.D. with the nth Fibonacci number},
Indagationes Mathematicae (2018).

\bibitem{SM} 
C.~Smyth, \emph{The terms in {L}ucas sequences divisible by their indices}, J.  Integer Seq. \textbf{13} (2010), no.~2, Article 10.2.4, 18.

\bibitem{SO}
L.~Somer, \emph{Divisibility of terms in {L}ucas sequences by their
  subscripts}, Applications of {F}ibonacci numbers, {V}ol. 5 ({S}t. {A}ndrews,
  1992), Kluwer Acad. Publ., Dordrecht, 1993, pp.~515--525.

\bibitem{T}
G.~Tenenbaum, \emph{Introduction to Analytic and Probabilistic Number Theory}, Cambridge U. P., 1995.
\end{thebibliography}

\end{document}